\documentclass[12pt]{amsart}
\usepackage{amssymb}
\usepackage{amsthm}
\usepackage{verbatim}
\usepackage[mathscr]{euscript}
\usepackage{xspace}
\usepackage{tikz}

\usepackage{epsfig}  		

\def\ds{\displaystyle}
\newcommand{\R}{\ensuremath{\mathbb{R}}\xspace}

\newcommand{\supp}{\ensuremath{\mbox{supp}\,}\xspace}

\newcommand{\be}{\begin{equation}}
\newcommand{\ee}{\end{equation}}

\newcommand{\bn}{\begin{equation*}}      
\newcommand{\en}{\end{equation*}}
\newcommand{\bea}{\begin{align}}
\newcommand{\eea}{\end{align}}
\newcommand{\bs}{\begin{split}}
\newcommand{\es}{\end{split}}
\newcommand{\beq}{\begin{equation}}
\newcommand{\eeq}{\end{equation}}

\renewcommand{\bar}{\overline}

\newcommand{\om}{\omega}

\newcommand{\al}{\alpha}
\newcommand{\ga}{\gamma}

\newcommand{\pvint}{P.V.\!\!\int}

\newtheorem{theorem}{Theorem}[section] % numbered like the section
\newtheorem{lemma}[theorem]{Lemma} % numbered like the theorems
\newtheorem{proposition}[theorem]{Proposition}

\newtheorem{remark}[theorem]{Remark}

\theoremstyle{definition} % styled differently... not italicized

\title{One-dimensional model equations for hyperbolic fluid flow}
\author{Tam Do}
\address{Tam Do, Rice University, Department of Mathematics-MS 136, Box 1892, Houston, TX 77251-1892}
\email{tam.do@rice.edu}
\author{Vu Hoang}
\address{Vu Hoang, Rice University, Department of Mathematics-MS 136, Box 1892, Houston, TX 77251-1892}
\email{Vu.Hoang@rice.edu}
\author{Maria Radosz}
\address{Maria Radosz, Rice University, Department of Mathematics-MS 136, Box 1892, Houston, TX 77251-1892}
\email{maria\_radosz@hotmail.com}
\author{Xiaoqian Xu}
\address{Xiaoqian Xu, Department of Mathematics, University of Wisconsin, Madison, WI 53706}
\email{xxu@math.wisc.edu}
\date{\today}

\begin{document}
\begin{abstract}
In this paper we study the singularity formation for two nonlocal 1D active scalar equations, focusing on
the hyperbolic flow scenario. Those 1D equations can be regarded as simplified models of some 2D fluid equations.
\end{abstract}

\maketitle
\section{Introduction}
The following transport equation
\begin{equation}\label{eq1}
\omega_t+u\cdot \nabla \omega=0.
\end{equation}
is a basic mathematical model in fluid dynamics. If $u$ depends on $\omega$, \eqref{eq1} is called an active scalar equation. 
The problem of deciding whether blowup can occur for smooth initial data becomes very
hard if the dependence of $\omega$ is nonlocal in space. 

The relationship expressing $u$ in terms of $\omega$ is commonly called Biot-Savart law.
We have the following examples in 2D: 
\begin{equation}\label{eq2}
u=\nabla^{\perp}(-\triangle)^{-1}\omega,
\end{equation}
where $\nabla^{\perp}=(-\partial_y,\partial_x)$ is the perpendicular gradient. Equations \eqref{eq1} and \eqref{eq2} are the vorticity form of 2D Euler equation.
When we take
$$
u=\nabla^{\perp}(-\Delta)^{-\frac{1}{2}}\omega,
$$
\eqref{eq1} becomes the surface quasi-geostrophic (SQG) equation, which has important applications in geophysics, or can be regarded
as a toy model for the 3D-Euler equations. For more details we refer to \cite{constantin1994formation}.

A question of great importance is whether solutions for these equations form singularities in finite time. A promising
new approach for the construction of singular solutions is to use the \emph{hyperbolic flow scenario}.
In \cite{HouLuo1}, \cite{HouLuo2}, such a scenario was proposed to obtain singular solutions for the 3D Euler equations, and in \cite{kiselev2013small},
the long-standing question of existence of solutions to the 2D Euler with double-exponential gradient growth was settled using
hyperbolic flow. 

The hyperbolic flow scenario in two dimensions can be explained in the following way. Consider e.g. a flow in the upper half-plane
$\{x_2 > 0 \}$. The essential properties required are (see Figure \ref{fig1} for an illustration):
\begin{itemize}
\item There is a stagnant point of the flow at one boundary point (e.g. the origin) for all times.
\item Along the boundary, the flow is essentially directed towards that point for all times.
\end{itemize}
Such flows can be created by imposing symmetry and other conditions on the initial data. For incompressible flows 
the stagnant point is a hyperbolic point of the velocity field, hence the name.

\begin{figure}[htbp]
\begin{center}
\includegraphics[scale=0.5]{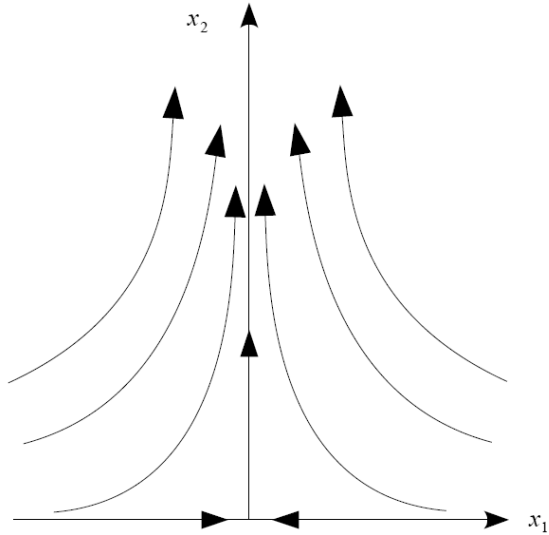}
\end{center}
\caption{Illustration of hyperbolic flow scenario in two dimensions.} \label{fig1}
\end{figure}

The scenario is a natural candidate for creating flows with strong gradient growth or finite-time blowup, since the fluid is compressed along the boundary. 
Due to non-linear and non-local interactions however, the flow remains hard to control, so a rigourous proof of blowup for the
3D Euler equations using hyperbolic flow remains a challenge. The crucial issue is to stabilize the scenario up to the singular time. 

One way to make progress in understanding and to gain insight into the hyperbolic blowup scenario is to study it in the context of
one-dimensional model equations. This was begun in \cite{CKY, sixAuthors}, where one-dimensional models for the 2D-Boussinesq 
and 3D axisymmetric Euler equations were introduced and blowup was proven.

One-dimensional models capturing other aspects of fluid dynamical equations have a long-standing tradition, one of the earliest being
the celebrated Constantin-Lax-Majda model \cite{CLM}. We refer to the introduction of \cite{silvestre2014transport} for 
a more thorough review of known one-dimensional model equations, and to \cite{sixAuthors} for discussion of the aspects relating
to the hyperbolic flow scenario.  

In this paper, we will study 1D models of (\ref{eq1}) on $\R$ with the following two choices of $u$:
\begin{eqnarray}
\label{eq3} u_x&=&H\omega,\\
\label{eq4} u&=&(-\Delta)^{-\frac{\alpha}{2}}\omega=-c_\alpha\int_{\R}|y-x|^{-(1-\alpha)}\omega(y,t)~ dy.
\end{eqnarray}
The choice \eqref{eq3} leads to a 1D analogue of the 2D Euler equation. This model is derived simply 
by restricting the dynamics to the boundary. In section \ref{sec:Euler} we give a brief heuristic argument which works by
assuming that $\om$ is concentrated in a small boundary layer.

We note that the model \eqref{eq3} was mentioned in \cite{sixAuthors}, where it was stated that \eqref{eq3} has properties
analogous to the 2D Euler equation, without giving details. In particular, in \cite{sixAuthors} a 1D model of the 2D Boussinesq equations (an extended version of \eqref{eq3}) was introduced and studied. One of our goals here is to validate the 1D model introduced in \cite{sixAuthors} in a setting where comparison with 2D results are available. The fact shown below, that the solutions to the model problem \eqref{eq3} behave similarly to the full 2D Euler case, provides support to the usefulness of the extended version of this model in \cite{sixAuthors} for getting insight into behavior of solutions to 2D Boussinesq system and 3D Euler equation.

The model defined by \eqref{eq4} is called $\alpha$-patch model and appears in \cite{dong2014one}, where also a viscosity term is present.
From the regularity standpoint, the $\alpha$-patch model is between 1D Euler $u_x = H \om$ and the C\'ordoba-C\'ordoba-Fontelos model $u=H\om$ (see \cite{CCF, silvestre2014transport}), which is an analogue of the SQG equation. These two models differ however from a geometric perspective, 
since the symmetry properties of the Biot-Savart laws are different. For the CCF model, the velocity field is odd for 
even $\om$, whereas \eqref{eq4} is odd for odd $\om$. It is important to choose
data with the right symmetry to make $u$ odd, and thus to create a stagnant point of the flow at the origin for all times.

We note that local existence and blowup results for \eqref{eq4} were given in \cite{dong2014one}, where also dissipation
is allowed. There the authors rely on a suitable Lyapunov function to show blowup, whereas we
emphasize the more geometric aspects in this paper. That is, we will be studying the analogue of the hyperbolic flow scenario for the above 1D models 
and show that this leads to natural and intuitive constructions of solutions with strong gradient growth and finite-time blowup. 

Another blowup result related to hyperbolic flow was recently proven by A. Kiselev, L. Ryzhik, Y. Yao and A. Zlato\v s \cite{KiselevLenyaYao} and concerns a 
$\alpha$-patch model in 2D for small $\alpha>0$.

\section{Euler 1D model}\label{sec:Euler}
\subsection{Heuristic derivation.}
Recall the 2D Euler equations in vorticity form
$$
\omega_t + u\cdot \nabla \omega = 0 
$$
where $u=\nabla^\perp (-\Delta)^{-1} \omega$.

We first indicate a simple heuristic motivation for the choice \eqref{eq3} (see also \cite{sixAuthors}). Consider the 2D Euler equation in
a half-space $\{x_2 \geq 0 \}$ and denote $\bar x = (x_1, -x_2)$. The $x_1$-component of the velocity (up to a normalization constant) 
for compactly supported vorticity $\om$ is given by
\beq
u_1(x, t) = -\int_{\R^2} \frac{(y_2-x_2)}{|y-x|^2} \om(y, t)~ dy
\eeq
where $\om$ has been extended to $\{ x_2 \leq 0 \}$ by odd reflection $(\om(\bar x, t)= -\om(x, t))$.

Suppose now that $\om$ is concentrated in a boundary layer of width $a>0$ and that $\om(x_1, x_2, t) = \om(x_1, t)$ in this boundary layer. Then a calculation gives
\beq\label{eq_u1}
u_1(x_1, 0, t) = -2 \int_{\R} \log\left( \frac{(y_1-x_2)^2+a^2}{(y_1-x_1)^2}\right) \om(y_1, t)~ dy_1.
\eeq
If we now retain only the singular part of the kernel $\log\left( \frac{z^2+a^2}{z^2}\right)\sim - 2 \log|z|$ and identify $u$ with $u_1$, we get
(dropping the constants) 
$$
u(x, t) =  \int_{\R}\log|y-x| \om(y, t)~ dy.
$$

So a reasonable 1D model is
\begin{equation}\label{1}
\omega_t+u\omega_x=0,\quad u_x=H\omega,\quad \omega(x,0)=\omega_0(x) \quad (x,t)\in \R\times [0,\infty).
\end{equation}
where $H$ is the Hilbert transform, using the convention 
$$
H \om(x, t) = \pvint\frac{\om(y, t)}{x-y}~ dy.
$$
For this model, we have the following local well-posedness property:
\begin{proposition}\label{local}
Given initial data $\omega_0\in H_0^{m}((0,1))$ with $m\geq 2$, there exists $T=T(\|\omega_0\|_{H_0^m})>0$ such that the system has a unique classical solution $\omega\in C([0,T];H^{m}_0)$.
\end{proposition}
The proof is standard so we skip it here.

An alternative argument to motivate \eqref{eq3} is to observe that the gradient of the 2D Euler velocity is given by a
zero-th order operator acting on $\om$. In one dimension, this leaves only the choice $u_x = c H\om$ or $u_x = c\om$,
$c$ being a nonzero constant. So we could also consider the model 
\begin{equation}\label{alternativeModel}
\omega_t+u\omega_x=0,\quad u_x=-\omega,\quad \omega(x,0)=\omega_0(x) \quad (x,t)\in \R\times [0,\infty).
\end{equation}
\eqref{alternativeModel} is however not a close analogue of 2D Euler (see Remark \ref{rem:alternativeModel}).

\subsection{Sharp a-priori bounds for gradient growth.}
We will first prove the global regularity of the solution to equation \eqref{1} by showing that $\omega_x$ can grow at most 
with double exponential rate in time. Then we will give an example of a smooth solution to (\ref{1}) where such growth of 
the gradient of $\omega$ is achieved, meaning the bound is sharp.

Due to the Biot-Savart law relating $u$ and $\omega$, the proof of an upper bound for $\|\omega_x(\cdot, t)\|_\infty$ is very similar to the proof for 
the full 2D Euler equations. For the reader's convenience, we give the proof. 
Recall first the definition of the H\"older norm
$$
|| \omega ||_{C^\alpha} = \sup_{|x-y|\leq 1, x\neq y} \frac{|\omega(x)-\omega(y)|}{|x-y|^\alpha}
$$
for compactly supported $\om$.

We will need an estimate on the Hilbert transform: 
\begin{lemma}\label{lem1} Let $0<\alpha<1$. 
Suppose $\supp(\omega) \subset [-D(t),D(t)]$ and assume
without loss of generality that $||\omega_0||_{L^{\infty}}=1$. Then
$$
\|u_x\|_\infty\leq C(\alpha)\left(1+|\log(D(t))|+\log (1+\|\omega\|_{C^{\alpha}})\right)
$$
\end{lemma}
\begin{proof}
For any $\delta>0$, we have
$$
\left|\int_{[-D(t),D(t)]\setminus (x-\delta,x+\delta)}\frac{\omega(y)}{x-y}~dy\right|\leq C\int_{\delta}^{D(t)}\frac{1}{y}~dy\leq C(|\log\delta|+|\log(D(t))|).
$$
Using the oddness of $\frac{1}{x}$, we have
$$
\left|\pvint_{x-\delta}^{x+\delta}\frac{\omega(y)}{x-y}~dy\right|=\left|\int_{x-\delta}^{x+\delta}\frac{1}{x-y}(\omega(y)-\omega(x))~dy\right|\leq C(\alpha)\|\omega(x,t)\|_{C^{\alpha}}\delta^{\alpha}.
$$
Choosing $\delta=\min\left\{1,(\frac{1}{\|\omega\|_{C^{\alpha}}})^{\frac{1}{\alpha}}\right\}$, we get the desired estimate of $\|u_x\|_\infty$.
\end{proof}

The following Lemma gives an estimate on $D(t)$.
\begin{lemma} \label{lem2}Suppose the support of $\omega_0$ is in $[-1,1]$ and $||\omega_0||_{L^{\infty}}=1$. Then the support of $\omega(x,t)$ will be inside $[-C\exp(Ce^{C t}),C\exp(Ce^{Ct})]$, for some universal constant $C>0$.
\end{lemma}
\begin{proof}
 Suppose $\operatorname{supp} \omega = [-D(t),D(t)]$. Then for any point $x$ inside of this interval, we have 
 $$
 |u(x)|\leq \int_{-D(t)}^{D(t)}|\log|x-y||~dy\leq C\int_{0}^{2D(t)}|\log|s||~ds\leq C D(t)(|\log(D(t))|+1).
 $$
By following the trajectory of the particle at $D(t)$,
$$
D'(t)\leq C D(t)(|\log(D(t))|+1).
$$
A simple argument using differential inequalities shows that $D(t)$ is always less 
than $z(t)$, where $z(t)$ is the solution of 
$$
z'(t) = C z(t)(\log z(t)+ 1), ~~z(0) = \min\{ D(0), 2 \}.
$$
This yields the double-exponential upper bound on $D(t)$.
\end{proof}

The following Theorem gives the double exponential upper bound for $\omega_x$.
\begin{theorem}  There is universal constant $C$ such that if $\omega_0$ is smooth, compactly supported with $\supp \om_0\subset [-1, 1]$ and $\|\om\|_{L^\infty}=1$,   
\beq\label{doubleExpBound}
\log(1+\|\omega_x\|_{L^{\infty}})\leq C\log (1+\|(\omega_0)_x\|_{L^{\infty}}) e^{Ct}\quad (t\geq 0).
\eeq
\end{theorem}
\begin{proof}
We follow the proof in \cite{kiselev2013small}. Let us denote the flow map corresponding to the evolution by $\Phi_t(x)$. Then
$$
\frac{\partial}{\partial t}\Phi_t(x) = u(\Phi_t(x),t),\quad\Phi_0(x)=x,
$$
and
$$
\left|\frac{\partial_t|\Phi_t(x)-\Phi_t(y)|}{|\Phi_t(x)-\Phi_t(y)|}\right|\leq ||u_x||_{L^{\infty}}.
$$
After integration, and by Lemma \ref{lem1} and Lemma \ref{lem2}, this gives
$$
f(t)^{-1}\leq \frac{|\Phi_t(x)-\Phi_t(y)|}{|x-y|}\leq f(t),
$$
where
$$
f(t)=\exp\left(C\int_0^t(1+\exp(Cs)+\log(1+||\omega_x||_{L^{\infty}}))~ds\right).
$$
This bound also holds for $\Phi_t^{-1}$. On the other hand,
$$
\|\omega_x\|_{L^{\infty}}=\sup_{x\ne y}\frac{|\omega_0(\Phi_t^{-1}(x))-\omega_0(\Phi_t^{-1}(y))|}{|x-y|}\leq \|(\omega_0)_x\|\sup_{x\ne y}\frac{|\Phi_t^{-1}(x)-\Phi_t^{-1}(y)|}{|x-y|}.
$$
Which means we have
$$
(1+\|\omega_x\|_{L^{\infty}})\leq (1+\|(\omega_0)_x\|_{L^{\infty}})\exp\left(C\int_0^t 1+\exp(Cs)+\log(1+\|\omega_x\|_{L^{\infty}})~ds\right),
$$
or
$$
\log(1+\|\omega_x\|_{L^{\infty}})\leq\log(1+\|(\omega_0)_x\|_{L^{\infty}})+C\exp(Ct)+C\int_0^t(1+\log(1+\|\omega_x\|_{L^{\infty}}))~ds.
$$
So $y(t):= \log(1+\|\omega_x\|_{L^{\infty}})$ satisfies the integral inequality   
$$
y'(t)\leq y(0) + C e^{C t} + \int_0^t (1+y(s))~ds
$$
and by the integral form Gronwall's inequality and some elementary manipulations,
we arrive at the bound $y(t) \leq C_1 y(0) e^{C_2 t}$. This yields the desired 
bound on $\|\om_x\|_{\infty}$. 
\end{proof}
\begin{remark}\label{rem:alternativeModel}
If we choose our Biot-Savart law to be $u_x=-\omega$, then from a modification of the above proof we get an 
exponential upper bound for $\|\omega_x\|_{L^{\infty}}$.
This is different from the 2D Euler equation, which suggests that \eqref{1} is a better analogue of the 2D Euler equation
than \eqref{alternativeModel}.  Moreover the equation \eqref{alternativeModel} also has different symmetry properties.
\end{remark}

Next we construct initial data $\omega_0$ such that $\|\omega_x(\cdot,t)\|_{L^{\infty}}$ grows with double-exponential rate,
proving the sharpness of the a-priori bound \eqref{doubleExpBound}. The hyperbolic flow scenario is created in the following way.
First, we require that the initial data $\om_0$ is odd with respect to the origin, and has compact support. 
By Proposition \ref{local}, the oddness is easily seen to be preserved by the evolution. Consequently, the velocity field (which is also an odd function) can be written as
\beq\label{oddBiotSavart}
u(x, t) = - x \int_0^\infty K\left(\frac{x}{y}\right) \frac{\om(y,t)}{y}~dy \quad (x > 0),
\eeq
where 
\beq\label{defK}
K(s) := \frac{1}{s}\log\left|\frac{s+1}{s-1}\right|.
\eeq
Note that the origin is a stagnant point of the flow for all times.
By taking $\om_0$ to be positive on the right, the direction of the flow is towards the origin. More precisely, 
$\om_0$ is defined as follows (see Figure \ref{fig2}):
\begin{itemize}
\item
Let $\omega_0$  be supported on $[-1,1]$, smooth and odd. Choose numbers $0<x_1(0) < 2 x_2(0) <1$ such that $M x_1(0)\le x_2(0)$, where $M$ will be determined later.
Require that $\omega_0$ is increasing on $[0,x_1(0)]$, decreasing on $[x_2(0),1]$ and identically 1 on $[x_1(0),x_2(0)]$.
\end{itemize} 
Using the earlier notation $\Phi_t$ for the flow map associated to $\eqref{1}$, let
\begin{align*}
x_1(t) &:= \Phi_t(x_1(0)) \\
x_2(t) &:= \Phi_t(x_2(0))
\end{align*}
It is easy to see that the general structure of $\omega_0$ will be preserved by the flow: For fixed $t$, $\omega(x,t)$ will be increasing on $[0,x_1(t)]$, decreasing on $[x_2(t),1]$ and identically 1 on $[x_1(t),x_2(t)]$. In fact, since $u(x,t)\le 0$ for $x\ge 0$, $x_1(t)$ and $x_2(t)$ will be moving towards the origin in time. We will show that the quantity $\ds\frac{x_2(t)}{x_1(t)}$ increases double exponentially in time. This is sufficient to conclude the desired growth of $\|\omega_x(\cdot, t)\|_{L^\infty}$.

\begin{figure}
\begin{center}
\includegraphics[scale=0.6]{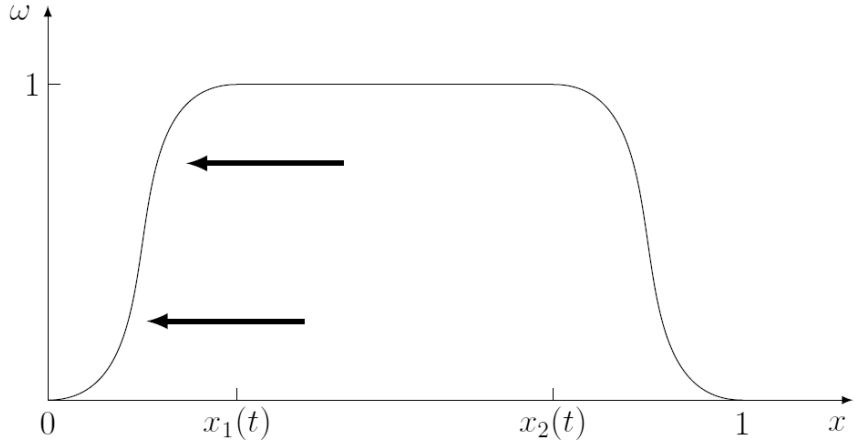}
\end{center}
\caption{Structure of $\om(x, t)$.}\label{fig2}
\end{figure}

\begin{theorem}\label{doubleExpGrowth} Assume our initial data is defined as above, then 
$$\log \frac{x_2(t)}{x_1(t)}\geq \log \frac{x_2(0)}{x_1(0)} \exp(Ct) \quad (t>0),$$
for some positive constant $C$. As a consequence, 
$$
\log \|\omega_x(\cdot,t)\|_{L^\infty}\ge C_1 \exp(C_2 t)\quad (t>0)
$$
for some $C_1, C_2 > 0$.
\end{theorem}

Theorem 2.4 quickly follows from the following Lemma:
\begin{lemma}\label{lem5} Suppose $1\geq x_2\geq 8x_1$. There are universal constants $C_0$ and $C_1$ so that 
$$
\frac{d}{dt}\left(\frac{x_2}{x_1}\right)\geq C_1 \frac{x_2}{x_1}\left(\log\left(\frac{x_2}{x_1}\right)-C_0\right).
$$
\end{lemma}
\begin{proof}
First observe
\begin{align*}
\frac{d}{dt}\left(\frac{x_2}{x_1}\right) &=\frac{x'_2x_1-x'_1x_2}{x^2_1} =\frac{u(x_2)x_1-u(x_1)x_2}{x_1^2} =\frac{x_2}{x_1}\left(\frac{u(x_2)}{x_2}-\frac{u(x_1)}{x_1}\right)\\
&=\frac{x_2}{x_1}\int_0^1 \left[K\left(\frac{x_1}{y}\right)-K\left(\frac{x_2}{y}\right)\right]\frac{\omega(y)}{y}~dy.
\end{align*}
We decompose the integral into 4 pieces which we will estimate separately:
\begin{align*}
&\int_0^1 \left[K\left(\frac{x_1}{y}\right)-K\left(\frac{x_2}{y}\right)\right]\frac{\omega(y)}{y}~dy\\
=&\int_0^{2x_1}+\int_{2x_1}^{\frac{1}{2}x_2}+\int_{\frac{1}{2}x_2}^{2x_2}+\int_{2x_2}^1\left[K\left(\frac{x_1}{y}\right)-K\left(\frac{x_2}{y}\right)\right]\frac{\omega(y)}{y}~dy\\
=& I+II+III+IV.
\end{align*}
For $I$, we use $0\le \om(y)\le 1$ and $2x_1 \le x_2\leq 1$:
\begin{align*}
0\le I&\le \int_0^{2x_1}\frac{1}{x_1}\log\frac{(x_1+y)}{|x_1-y|}~dy
+\int_0^{2x_1}\frac{1}{x_2}\log\frac{(x_2+y)}{|x_2-y|}~dy\\
&=\frac{1}{x_1}3x_1\log 3+ \frac{1}{x_2}\left[2x_1 \log\frac{1+\frac{2x_1}{x_2}}{1-\frac{2x_1}{x_2}}+x_2\log\left(1-\frac{2x_1}{x_2}\right)+x_2\log\left(1+\frac{2x_1}{x_2}\right)\right]\\
&\leq 3\log 3 + 2\log 2 .
\end{align*}
Using the fact that $K(s)$ is increasing in $[0,1)$ and decreasing in $(1,\infty]$ and that $\om(y)=1$ for $y\in(2x_1,\frac{1}{2}x_2)$ we get
\begin{align*}
II&=\int_{2x_1}^{\frac{1}{2}x_2}\left[K\left(\frac{x_1}{y}\right)-K\left(\frac{x_2}{y}\right)\right]\frac{\omega(y)}{y}~dy \geq\int_{2x_1}^{\frac{1}{2}x_2}(2-\frac{1}{2}\log(3))\frac{1}{y}~dy\\
&= (2-\frac{1}{2}\log(3))\log\left(\frac{x_2}{x_1}\right)-C.
\end{align*}
Using the positivity of $K$, 
\begin{align*}
III&=\int_{\frac{1}{2}x_2}^{2x_2}\left[K\left(\frac{x_1}{y}\right)-K\left(\frac{x_2}{y}\right)\right]\frac{\omega(y)}{y}~dy \geq -\int_{\frac{1}{2}x_2}^{2x_2}K\left(\frac{x_2}{y}\right)\omega(y)\frac{1}{y}~dy\\
&\geq-\int_{\frac{1}{2}}^2\frac{1}{s^2}\log\frac{|s+1|}{|s-1|}~ds \geq -C.
\end{align*}
We estimate $IV$ in the following way, using that $\om(y)\le 1$ and $\frac{x_1}{y}\le \frac{x_2}{y}\le1$ for $2x_2\le y \le 1$:
\begin{align*}
|IV|&=\left|\int_{2x_2}^1\left[K\left(\frac{x_1}{y}\right)-K\left(\frac{x_2}{y}\right)\right]\frac{\omega(y)}{y}~dy \right| \leq \int_{2x_2}^1\left[K\left(\frac{x_2}{y}\right)-K\left(\frac{x_1}{y}\right)\right]\frac{1}{y}~dy \\
&\leq \int_{2x_2}^1\frac{1}{x_2}\log\frac{y+x_2}{y-x_2}~dy-\int_{2x_2}^1\frac{1}{x_1}\log\frac{y+x_1}{y-x_1}\\
&=(i)-(ii).
\end{align*}
We can compute $(i)$ directly and get
\begin{align*}
(i)=\frac{1}{x_2}\log \frac{1+x_2}{1-x_2}+\log(1+x_2)(1-x_2)-2\log(x_2)-3\log(3).
\end{align*}
Similarly, for $(ii)$, we have
\begin{align*}
(ii)&=\frac{1}{x_1}\log\frac{1+x_1}{1-x_1}+\log(x_1+1)(1-x_1)-2\frac{x_2}{x_1}\log\frac{2x_2+x_1}{2x_2-x_1}-\log(2x_2+x_1)(2x_2-x_1).
\end{align*}
Note that in the expressions for $(i)$ and $(ii)$, all terms can be bounded by universal constants
except for $-2\log(x_2)$ and $\log(2x_2+x_1)(2x_2-x_1)$.
However, using $x_1<x_2$, we get
\begin{align*}
|IV| \le C -2\log(x_2)+\log(2x_2+x_1)(2x_2-x_1) =C +\log\left(4-\left(\frac{x_1}{x_2}\right)^2\right) \le C.
\end{align*}
\end{proof}
The proof of Theorem \ref{doubleExpGrowth} is now completed as follows: choose $M>8$ so large such that $\frac{1}{2}\log(M)-C_0 \geq 0$. We have thus 
$\frac{1}{2}\log\left(\frac{x_2^0}{x_1^0}\right)-C_0 \geq 0.$
From Lemma \ref{lem5} it follows that $\frac{x_2(t)}{x_1(t)}$ is growing in time and that we have
$$
\frac{d}{dt}\left(\frac{x_2}{x_1}\right)\geq \frac{C_1}{2} \frac{x_2}{x_1}\log\left(\frac{x_2}{x_1}\right),
$$ 
or $\frac{d}{dt}\log \left(\frac{x_2}{x_1}\right)\geq \frac{C_1}{2} \log\left(\frac{x_2}{x_1}\right)$
for all times. This clearly implies that $\frac{x_2}{x_1}$ grows double-exponentially.

\begin{remark}
In \cite{kiselev2013small}, the Biot-Savart law is decomposed into a main contribution and an error term. In our case \eqref{oddBiotSavart}, 
the main contribution would be
\beq\label{hyperbolicApprox}
-x\int_x^\infty \frac{\omega(y)}{y}~dy.
\eeq
If we replace \eqref{oddBiotSavart}  by \eqref{hyperbolicApprox}, then double-exponential growth of $\frac{x_2}{x_1}$ can be proven by a straightforward argument. In this case, the computation for the estimate in Lemma \ref{lem5} becomes much easier.
\end{remark}

\section{$\alpha$-patch 1D model}
In this section, we consider the 1D model equation
\beq\label{eq:model}
\omega_t+u\omega_x=0
\eeq
with a different Biot-Savart law
\beq
u(x,t)= (-\Delta)^{-\alpha/2}\omega(x,t) =-c_\alpha\int_{\R}|y-x|^{-(1-\alpha)}\omega(y,t)~dy,\qquad\alpha\in(0,1)
\eeq
For convenience, we will assume the constant $c_\alpha$ associated with the fractional Laplacian is $1$, and we write $\ga = 1-\al$.

This problem has been studied in \cite{dong2014one}, where local existence and uniqueness results for smooth initial data are proven. From these, we can show 
that this equation preserves oddness and $u(0,t)=0$ holds with odd initial datum. For odd data, we can write
\beq
u(x,t)=-\int_0^\infty k(x,y)\omega(y,t)~dy
\eeq
where $k(x,y)=|y-x|^{-\ga}-|y+x|^{-\ga}$. Note that $k(x,y)\ge 0$ for $x\neq y\in (0,\infty)$.

Following similar ideas as for 1D Euler, we specifiy our initial data $\omega_0$ as follows:
\begin{itemize}
\item
Pick $0<x_1(0), x_2(0)$ with $M x_1(0)<x_2(0)$. Let $\omega_0$ be smooth, odd, $\omega_0(x)\geq 0$ for $x>0$ and have its support in $[-2 x_2(0), 2 x_2(0)]$.
$M > 1$ is to be chosen below. Moreover, let $\omega_0$ be bounded by $1$, smoothly increasing in the interval $[0, x_1(0)]$ and $\omega_0 = 1$ between $x_1(0)$ and $x_2(0)$.
\end{itemize}
As long as the solution remains smooth, the general structure of the solution does not change.
Let $x_1(t), x_2(t)$ be again the position of the particles starting at $x_1(0), x_2(0)$. 
\begin{theorem}\label{blowupPatch}
There exist a choice of $x_1(0), x_2(0), M$ and a time $T>0$ such that the smooth solution of   
$(\ref{eq:model})$ for the above initial data cannot be continued beyond $T$. Provided the solution 
remains smooth on the time interval $[0, T)$, the particle starting at $x_1(0)$ reaches the origin at time $t=T$, i.e.
\beq
\lim_{t\to T} x_1(t) = 0.
\eeq
In this sense, the solution forms a ``shock".
\end{theorem}
\begin{remark}
In \cite{dong2014one}, the existence of blowup solutions to $(\ref{eq:model})$ is shown using 
energy methods. The advantage is that they are able to include a dissipation term. However,  
it is difficult to see the geometric blowup mechanism clearly using energy methods. 
Our proof for the inviscid case uses the dynamics of the solution and gives a more intuitive picture of the blowup,
and is easily generalized to other even kernels having the same singular behavior.
\end{remark}
In the rest of this section, we will prove Theorem \ref{blowupPatch}. So assume that for arbitrary choice
of $x_1(0), x_2(0), M$, we have a smooth solution $\om$ defined for all times $t\in [0,\infty)$.

First of all, we track the movement of the particle starting at $x_1(0)$, which is the following Lemma.
\begin{lemma}
There exists a universal constant $M>2$ so that if $M x_1(t)\leq x_2(t)$, the velocity at $x_1(t)$ will satisfy 
\begin{equation}\label{blup}
u(x_1(t),t)\leq -Cx_1(t)^{1-\ga},
\end{equation}
for some universal constant $C$.
\end{lemma}
 \begin{proof}
    Let $u_1=u(x_1(t),t)$. Since $k,\omega \ge 0$ on $(0,\infty)$
\begin{align*}
-u_1  &\ge \int_{2x_1}^{x_2}k(x_1,y)~dy\\
   &=c_\ga\left[-(x_2+x_1)^{1-\ga}+(x_2-x_1)^{1-\ga}+(3x_1)^{1-\ga}-x_1^{1-\ga}\right]\\
  &=c_\ga\left[(3^{1-\ga}-1)x_1^{1-\ga}+(x_2-x_1)^{1-\ga}-(x_2+x_1)^{1-\ga}\right]\\
  &=c_\ga x_1^{1-\ga}\left[(3^{1-\ga}-1)+\frac{1}{x_1^{1-\ga}}\left((x_2-x_1)^{1-\ga}-(x_2+x_1)^{1-\ga}\right)\right]
\end{align*}
for some constant $c_\ga>0$. Note that $(3^{1-\ga}-1)>0$. We can write
\begin{align*}
\frac{1}{x_1^{1-\ga}}(x_2-x_1)^{1-\ga}-(x_2+x_1)^{1-\ga}
& = \frac{x_2^{1-\ga}}{x_1^{1-\ga}}\left[\left(1-\frac{x_1}{x_2}\right)^{1-\ga}-\left(1+\frac{x_1}{x_2}\right)^{1-\ga}\right]\\
& =: \frac{x_2^{1-\ga}}{x_1^{1-\ga}} f(x_1/x_2).
\end{align*}
There exists a constant $C>0$ with $|f(x_1/x_2)|\leq C|x_1/x_2|$ for $|x_1/x_2| \leq 1/2$, and so
\begin{align*}
-u_1 \ge c_\ga x_1^{1-\ga}\left[(3^{1-\ga}-1) - C M^{-\ga}\right]
\end{align*}
if $M x_1(t) \leq x_2(t)$. Now choose $M$ large enough so that $C M^{-\ga}$ is smaller than the number $\frac{1}{2}(3^{1-\ga}-1)$.
\end{proof}
This estimate of velocity field will lead to a blowup in finite time, provided we can show $M x_1(t)\leq x_2(t)$. More precisely,  
$$
\frac{d}{dt}x_1(t)= u(x_1)\leq - C x_1^{1-\ga}, 
$$
implying
$$
x_1\leq C (x_1(0)^{\ga} - C t)^{\frac{1}{\ga}}.
$$
This shows that no later than $T_0 := C^{-1} x_1(0)^{\ga}$, the particle $x_1(t)$ will reach the origin, and the solution cannot be continued smoothly. Note 
that $T_0$ does not depend on $x_2(0)$.

It remains therefore to control the motion of $x_2(t)$, concluding the proof.
\begin{lemma}
For $x_2(0)$ large enough, $M x_1(t) < x_2(t)$ for $t\in [0, T_0)$.
\end{lemma}
\begin{proof}
We write $u(x_2(t),t)=u_2$. Observe that the support of 
$\om(\cdot, t)$ is always contained in $[-2 x_2(0), 2 x_2(0)]$ 
because of $u(x, t) \leq 0$ for $x>0$.

Next we find an upper bound on $u_2$:
\beq
|u_2(t)|\leq \int_{-2 x_2(0)}^{2 x_2(0)} |y-x|^{-\ga} \leq C x_2(0)^{1-\ga}. 
\eeq
Hence,
\beq
x_2(t) \geq x_2(0) - \int_0^{T_0} |u_2(s)|~ds\geq x_2(0) (1 - C x_2(0)^{-\ga} T_0).
\eeq
Now choose $x_2(0)$ so large that $M x_1(0) < x_2(0) (1 - C x_2(0)^{-\ga} T_0)$.
But then
$$M x_1(t)\leq M x_1(0) < x_2(0) (1 - C x_2(0)^{-\ga} T_0) \leq x_2(t),$$
giving the statement of the Lemma.
\end{proof}

\subsection*{Acknowledgments.} The authors would like to thank Prof.~Alexander Kiselev for helpful discussions. TD would like to acknowledge the support of NSF grant 1453199. VH expresses his gratitude to
the German Research Foundation (DFG) for continued support through grants FOR 5156/1-1 and FOR 5156/1-2. VH also acknowledges partial support by NSF grant
NSF-DMS 1412023. XX would like to acknowledge the support of NSF grant 1535653.

\bibliographystyle{plain}

\begin{thebibliography}{10}
\bibitem{sixAuthors}
Kyudong Choi, Thomas~Y. Hou, Alexander Kiselev, Guo Luo, Vladimir
  \v{S}ver\'{a}k, and Yao Yao.
\newblock On the finite-time blowup of a 1d model for the 3d axisymmetric
  {Euler} equations.
\newblock {\em arXiv:1407.4776}, 2014.

\bibitem{CKY}
Kyudong Choi, Alexander Kiselev, and Yao Yao.
\newblock Finite time blow up for a 1d model of 2d {Boussinesq} system.
\newblock {\em Comm. Math. Phys.}, 334(3):1667--1679, 2015.

\bibitem{CLM}
P.~Constantin, P.D. Lax, and A.~Majda.
\newblock A simple one-dimensional model for the three-dimensional vorticity
  equation.
\newblock {\em Comm. Pure Appl. Math.}, 38:715--724, 1985.

\bibitem{constantin1994formation}
Peter Constantin, Andrew~J Majda, and Esteban Tabak.
\newblock Formation of strong fronts in the 2-d quasigeostrophic thermal active
  scalar.
\newblock {\em Nonlinearity}, 7(6):1495--1533, 1994.

\bibitem{CCF}
A.~C\'ordoba, D.~C\'ordoba, and M.A. Fontelos.
\newblock Formation of singularities for a transport equation with nonlocal
  velocity.
\newblock {\em Ann. of Math.(2)}, 162(3):1377--1389, 2005.

\bibitem{dong2014one}
Hongjie Dong and Dong Li.
\newblock On a one-dimensional $\alpha$-patch model with nonlocal drift and
  fractional dissipation.
\newblock {\em Trans. Amer. Math. Soc.}, 366(4):2041--2061, 2014.

\bibitem{HouLuo1}
Thomas~Y. Hou and Guo Luo.
\newblock Toward the finite-time blowup of the 3d axisymmetric {Euler}
  equations: A numerical investigation.
\newblock {\em Multiscale Model. Simul.}, 12(4):1722--1776, 2014.

\bibitem{HouLuo2}
Thomas~Y. Hou and Guo Luo.
\newblock Potentially singular solutions of the 3D axisymmetric Euler equations.
\newblock {\em PNAS}, vol. 111 no. 36, 12968-12973,
\emph{DOI 10.1073/pnas.1405238111.}

\bibitem{KiselevLenyaYao}
Alexander Kiselev, Lenya Ryzhik, Yao Yao, and Andrej Zlato\v{s}.
\newblock Finite-time singularity formation for the modified SQG equation
\newblock {\tt arXiv:1508.07613v1}, 2015.

\bibitem{kiselev2013small}
Alexander Kiselev and Vladimir \v{S}ver\'{a}k.
\newblock Small scale creation for solutions of the incompressible two
  dimensional {Euler} equation.
\newblock {\em Ann. of Math.(2)}, 180(3):1205--1220, 2014.

\bibitem{silvestre2014transport}
Luis Silvestre and Vlad Vicol.
\newblock On a transport equation with nonlocal drift.
\newblock Trans. Amer. Math. Soc. 368, 6159--6188, 2016.
\end{thebibliography}

\end{document}